\documentclass[a4paper, 12pt]{article}
\usepackage{}

\usepackage{mathrsfs}
\usepackage{epsfig}
\usepackage{amsmath}
\usepackage{amssymb}
\usepackage{latexsym}
\usepackage{amsfonts}
\usepackage{amsthm}

\usepackage{bbm}

\usepackage[numbers,sort&compress]{natbib}
\usepackage{graphicx}
\ifx\pdfoutput\undefined  \else
 \fi

\usepackage[centerlast]{caption2}

\usepackage{color}

\usepackage{txfonts}
\usepackage{amssymb}
\usepackage{mathrsfs}
\usepackage{amsfonts}

\newtheorem{theorem}{Theorem}[section]
\newtheorem{lemma}[theorem]{Lemma}

\newtheorem{prop}[theorem]{Proposition}

\newtheorem{remark}[theorem]{Remark}
\newtheorem{definition}[theorem]{Definition}
\newtheorem{exam}[theorem]{Example}

\newtheorem{question}[theorem]{Question}

\usepackage{latexsym,amssymb}
\usepackage{amsmath}

\begin{document}

\title{{The universal zero-sum invariant and weighted zero-sum for infinite abelian groups II}}
\author{
Guoqing Wang\\
\small{School of
Mathematical Sciences, Tiangong University, Tianjin, 300387, P. R. China}\\
\small{Email: gqwang1979@aliyun.com}
\\
}
\date{}
\maketitle

\begin{abstract}
Let $G$ be an abelian group, and let $\mathcal F (G)$ be the free
commutative monoid with basis $G$, and $\mathcal A (G)$ the set consisting of all minimal zero-sum subsequences over $G$. For any subset $\Omega \subset \mathcal F (G)$, we define the universal zero-sum invariant ${\mathsf d}_{\Omega}(G)$ as the minimal positive integer $\ell$ such that
every sequence $T$ over $G$ of length $\ell$ contains a subsequence lying in $\Omega$.  The classical Davenport constant ${\rm D}(G)$ for $G$ can also be written as ${\mathsf d}_{\mathcal A (G)}(G)$. We give a complete classification of all finite abelian groups for which $\mathcal A(G)$ is a minimal set to represent the Davenport constant.

We also investigate  the weighted Davenport constant over abelian groups (which may be infinite).  Let $F$ and $G$ be abelian groups, and let $\Psi \subseteq \mathrm{Hom}(F,G)$ denote a weight set.
We reinterpret the weighted Davenport constant $D_{\Psi}(G)$ in terms of coverings of Cartesian powers $F^n$ by kernels of induced homomorphisms arising from tuples in $\Psi^n$; these homomorphisms are naturally linked to coproducts in the category of abelian groups.
This motivates the notion of kernel-cover compactness, a property characterizing when such kernel coverings admit finite subcovers. We establish a correspondence between weighted zero-sum invariants and kernel-cover structures, where the bound  $D_{\Psi}(G)\le n$ is equivalent to a canonical kernel-cover property on $F^n$.   We further study finite reduction phenomena for infinite weight sets and provide sufficient conditions ensuring uniform kernel-cover compactness.
The present work constitutes a follow-up to [G. Wang, Comm. Algebra, 2025].
\end{abstract}


\section {Introduction}

Let $G$ be a finite abelian group. The Davenport constant ${\rm D}(G)$ is a core invariant in zero-sum theory,
defined as the
minimal positive integer $\ell$ such that every sequence of $\ell$ terms from $G$
contains a nonempty subsequence whose sum equals the identity element of $G$.
 This invariant originates from ideal class groups in algebraic number theory (see \cite{Olson1}) and has deep connections with factorization theory (see
\cite{CziszterDoGerolding, GRuzsa, GH}) as well as combinatorial structures in algebra (see \cite{Grynkiewiczmono}).
Over the past decades, extensive research motivated by this constant has evolved zero-sum problems into a rich subfield of combinatorial number theory. We refer readers to \cite{GaoGeroldingersurvey} for a comprehensive survey.

Similar to the Davenport constant, zero-sum problems over finite abelian groups investigate the existence of substructures satisfying prescribed additive properties within combinatorial objects such as sets, sequences (multisets), graphs and hypergraphs. Within this framework, the existence of additive substructures is encoded by additive invariants collectively referred to as zero-sum constants. Over the past sixty years, many zero-sum constants have been introduced stemming from distinct research backgrounds; several of these originate from other branches of discrete mathematics, finite geometry included. To unify disparate zero-sum invariants, a general unifying framework was proposed in \cite{GaoLiPengWang}, formalized as Definition \ref{definition domenga}. Readers may consult \cite{GaoHong1,GaoHong2} for further related work.

Let $\mathcal F(G)$ be the free commutative monoid generated by $G$, and let $\mathcal B(G)$ denote the submonoid consisting of all zero-sum sequences over $G$.

\begin{definition} \label{definition domenga} \cite{GaoLiPengWang}
For any nonempty subset $\Omega \subset \mathcal F (G)$, we define the universal zero-sum invariant ${\mathsf d}_{\Omega}(G)$ to be the minimal positive integer $\ell$ (set ${\mathsf d}_{\Omega}(G)=\infty$ if no such $\ell$ exists) such that
every sequence over $G$ of length $\ell$ admits a subsequence contained in $\Omega$.
\end{definition}

In terms of this universal zero-sum invariant, the Davenport constant satisfies
$${\rm D}(G)={\mathsf d}_{\mathcal A(G)}(G),$$
where $\mathcal A(G)$ denotes the set of all minimal zero-sum sequences over $G$ (see \cite{GH}). This leads naturally to the following basic question.

\begin{question}\label{question} \cite{WangcommInAlgebra}
Is $\Omega=\mathcal A (G)$ minimal with respect to ${\mathsf d}_{\Omega}(G)={\rm D}(G)$, i.e., does there exist $\Omega'\subsetneq\mathcal A (G)$ such that ${\mathsf d}_{\Omega'}(G)={\rm D}(G)$?
\end{question}

This question is tightly linked to structural problems within zero-sum theory and has connections with some classical zero-sum problems. For instance, the well-known Lemke-Kleitman index conjecture \cite{LemkeKleitman} for finite cyclic groups can be restated as a statement concerning minimal sets to represent Davenport constant (see Conjecture 1.6 in \cite{WangcommInAlgebra}).

A partial answer of Question \ref{question} was established in \cite{WangcommInAlgebra}, stated as
Theorem A below.

\noindent {\bf Theorem A.} \ {Let $G$ be a finite nonzero abelian group with $\exp(G)\neq 3$. Then $\mathcal A (G)$ is a minimal set with respect to the Davenport constant ${\rm D}(G)$ if and only if one of the following conditions holds: (i)  $G\cong \mathbb{Z}_4$; (ii) $G\cong \mathbb{Z}_5$; (iii) $G\cong \mathbb{Z}_2^r$ for some $r\geq 1$.}

The remaining case $\exp(G)=3$ is still not fully understood in general. In Section 3, we conduct a thorough analysis of this setting and classify all finite abelian groups for which $\mathcal A(G)$ is minimal; the main result was stated as Theorem \ref{theorem in group exp(G)geq 6}.

On a separate line of research, weighted analogues of the Davenport constant, formulated within the general framework of group homomorphisms over finite abelian groups, were introduced by Yuan and Zeng (see \cite{ZengYuanDiscrete}). This invariant was originally investigated over finite abelian groups to establish weighted generalizations of classical Gao's Theorem (see \cite{Grynkiewiczmono}, Chapter 16), and its definition extends naturally to infinite abelian groups. This weighted Davenport constant and weighted universal zero-sum invariant have also been studied for arbitrary infinite abelian groups by the author in \cite{WangcommInAlgebra}. We note that the weighted Davenport constant is a specific instance of the weighted universal zero-sum invariant; readers may refer to \cite{WangcommInAlgebra} for extensive related work. In this paper, we restrict our attention exclusively to the weighted Davenport constant.

Let $F$ and $G$ be abelian groups and let $\Psi \subseteq \mathrm{Hom}(F,G)$ be nonempty. The $\Psi$-weighted Davenport constant ${\rm D}_\Psi(G)$ is defined as the minimal positive integer $\ell$ such that every sequence over $F$ of length $\ell$ admits a nontrivial $\Psi$-weighted zero-sum relation.
For infinite abelian groups $F$ and $G$, a core research problem is to characterize the conditions under which ${\rm D}_\Psi(G)$ is finite. When the weight set $\Psi$ is finite, this problem has been fully resolved in \cite{WangcommInAlgebra}. That paper proves that the finiteness of ${\rm D}_\Psi(G)$ is equivalent to $F$ admitting a finite covering by cosets of the kernels of homomorphisms from $\Psi$. Using the Neumann Covering Theorem for groups, the author derived necessary and sufficient conditions for ${\rm D}_{\Psi}(G)<\infty$ in terms of the weight set $\Psi$ (for finite $\Psi$), and obtained the following theorem.

\noindent{\bf Theorem B.} \cite{WangcommInAlgebra} {\sl Let $F, G$ be abelian groups, and let $\Psi$ be a nonempty finite subset of  ${\rm Hom}(F, G)$. Then the following conditions are equivalent.

(i) ${\rm D}_{\Psi}(G)<\infty$;

(ii) There exists a finite cover of the group $F$ by left cosets of ${\rm Ker} \ \psi$ ($\psi\in \Psi$), i.e.,
$$F=\bigcup\limits_{\psi\in \Psi}\left(\bigcup\limits_{i=1}^{k_{\psi}}(a_{\psi, i}+ {\rm Ker}  \ \psi)\right),$$ where $k_{\psi}\in \mathbb{N}$ and $a_{\psi, 1}, \ldots, a_{\psi, k_{\psi}}\in F$ for each $\psi\in \Psi$;

(iii) $|{\rm Im} \ \tau|<\infty$ for some $\tau\in \Psi$.}

However, when \(\Psi\) is infinite, the finite-weight coset-cover
characterization no longer directly applies.  In particular, even if
\({\rm D}_{\Psi}(G)<\infty\), the corresponding cover of \(F^n\) by kernels
of induced homomorphisms need not admit a finite subcover.

The main idea of this paper is to reinterpret the weighted Davenport
constant through the kernels of induced universal homomorphisms.  Given
$$
\boldsymbol{\eta}=(\psi_1,\ldots,\psi_n)
\in(\Psi\cup\{\mathbf 0\})^n,
$$
the universal property of direct sums gives an induced homomorphism
$$
\operatorname{ind}[\boldsymbol{\eta}]:F^n\to G,
\qquad
\operatorname{ind}[\boldsymbol{\eta}](x_1,\ldots,x_n)
=
\sum_{i=1}^{n}\psi_i(x_i).
$$
Thus weighted zero-sum relations over sequences of length \(n\) are
encoded by membership in kernels of such induced homomorphisms.

Our first main result shows that
${\rm D}_{\Psi}(G)\leq n$
if and only if \(F^n\) is covered by the kernels of the nonzero induced
homomorphisms arising from tuples in $(\Psi\cup\{\mathbf 0\})^n.$
This gives a geometric formulation of the weighted Davenport constant as
a kernel-cover threshold condition.

Motivated by this interpretation, we introduce the notion of
\emph{kernel-cover compactness}.  This is a closed-cover finiteness
property for the distinguished family of induced kernels.  It is not
ordinary topological compactness, but a compactness-type property adapted
to the kernel geometry of the weighted Davenport problem.  We prove that
finite subcovers of the canonical kernel cover are equivalent to finite
reduction of the weight set \(\Psi\).  We also give structural sufficient
conditions, including the finite effective quotient condition
$$
F/N_{\Psi}\ \text{finite},
\qquad
N_{\Psi}=\bigcap_{\psi\in\Psi}\ker\psi,
$$
which implies uniform kernel-cover compactness, i.e., every subset of $F^n$ is compact with respect to covers by the distinguished family
of induced kernels.

\section {Notation}

We use the standard notation and terminology for sequences over abelian
groups, following \cite{GaoGeroldingersurvey} and [\cite{GH}, Chapter 5]. Let ${\cal F}(G)$
be the free commutative monoid, multiplicatively written, with basis
$G$. By $T\in {\cal F}(G)$, we mean that $T$ is a finite unordered sequence over $G$, with
repetition allowed.

Denote $[x,y]=\{z\in \mathbb{Z}: x\leq z\leq y\}$ for integers $x,y\in \mathbb{Z}$. Say
$T=a_1a_2\cdot\ldots\cdot a_{\ell}$, where $a_i\in G$ for $i\in [1,\ell]$.
The sequence $T$ can be also denoted as  $\prod\limits_{a\in G}a^{{\rm v}_a(T)},$ where ${\rm v}_a(T)$ is a nonnegative integer and
means that the element $a$ occurs ${\rm v}_a(T)$ times in the sequence $T$. By $|T|$ we denote the length of the sequence, i.e., $|T|=\sum\limits_{a\in G}{\rm v}_a(T)=\ell.$ By $\varepsilon$ we denote the empty sequence in ${\cal F}(G)$ with $|\varepsilon|=0$.
By ${\rm supp}(T)$  we denote the set consisting of all distinct
elements which occur in $T$. We call $T'$
a subsequence of $T$ if ${\rm v}_a(T')\leq {\rm v}_a(T)$, for each element $a\in G,$ denoted by $T'\mid T,$ moreover, we write $T^{''}=T\cdot  T'^{-1}$ to mean the unique subsequence of $T$ with $T'\cdot T^{''}=T$.
Let $\sigma(T)=a_1+\cdots +a_{\ell}$ be the sum of all terms from $T$. In particular, we adopt the convention that $\sigma(\varepsilon)=0_G$.
Let $\Sigma(T)$ be the set consisting of the elements of $G$ that can be represented as a sum of one or more terms from $T$, i.e., $\Sigma(T)=\{\sigma(T'): T' \mbox{ is taken over all nonempty subsequences of }T\}$.
A sequence $T'$ is said to be equal to $T$, denoted $T'=T$, if ${\rm v}_a(T')={\rm v}_a(T)\ \ \mbox{for each element}\ \ a\in G.$ We call $T$ a {\bf zero-sum} sequence if $\sigma(T)=0$, and furthermore, call $T$ a {\bf minimal zero-sum} sequence, if $T$ is a nonempty zero-sum sequence and $T$ contains no nonempty proper ($\neq T$) zero-sum subsequence. We call $T$ a {\bf zero-sum free} sequence, if $T$ contains no nonempty zero-sum subsequence.

\begin{definition}\label{Definition minimal set}  (see \cite{GaoLiPengWang}, Section 4) Let $G$ be a finite abelian group, and let $t>0$ be an integer.  A set $\Omega\subset \mathcal B (G)$ is called {\bf minimal} with respect to $t$ provided that ${\mathsf d}_{\Omega}(G)=t$ and ${\mathsf d}_{\Omega'}(G)\neq t$ for any subset $\Omega'\subsetneq\Omega$.
In particular, if $G$ is finite and $t={\rm D}(G)$ then we just call $\Omega$ a minimal set for short.
\end{definition}

\begin{remark}\label{remark} By definition, we have that ${\mathsf d}_{\Omega_1}(G)\geq {\mathsf d}_{\Omega_2}(G)$ for any $\Omega_1\subset \Omega_2$.
\end{remark}

Let $F$ and $G$ be abelian groups, and let ${\rm Hom}(F, G)$ be the group consisting of all homomorphisms from $F$ to $G$. Let $\Psi$ be a nonempty subset of ${\rm Hom}(F, G)$.
Let $T\in \mathcal{F}(F)$.
Let $V=a_1\cdot\ldots\cdot a_k$ be a subsequence of $T$ ($k\geq 0$), and let $\psi_1,\ldots, \psi_{k}\in \Psi$ (not necessarily distinct). The sequence $\prod\limits_{i\in [1,k]}\psi_i(a_i)$ is called a $\Psi$-subsequence of $T$ (noticing that $\prod\limits_{i\in [1,k]}\psi_i(a_i)$ is a sequence over $G$). In particular, $\prod\limits_{i\in [1,k]}\psi_i(a_i)$ is called a {\bf $\Psi$-zero-sum} subsequence of $T$ if $\sum\limits_{i\in [1,k]}\psi_i(a_i)=0_G$. If $T$ has no nonempty $\Psi$-zero-sum subsequence, we say $T$ is  $\Psi$-zero-sum free.  To make the above definition more precise, we restate them as follows.
Denote
$$\Psi(V)=\{L\in \mathcal{F}(G): L=\psi_1(a_1)\cdot\ldots\cdot \psi_k(a_k) \mbox{ for some } \psi_1,\ldots,\psi_k\in \Psi\}.$$
Any sequence $W\in \bigcup\limits_{V\mid T} \Psi(V)$ is said to be a $\Psi$-subsequence of $T$, in particular, if $\sigma(W)=0_G$ then $W$ is called a $\Psi$-zero-sum subsequence of $T$. If $(\bigcup\limits_{V\mid T} \Psi(V))\cap \mathcal B (G)=\{\varepsilon\}$ then, the sequence $T$ is called $\Psi$-zero-sum free.

Note that for the case $F=G$ and $\Psi$ consisting of only the automorphism $1_G$, the terminology $\Psi$-zero-sum ($\Psi$-zero-sum free) is consistent to zero-sum (zero-sum free).

We now recall the definition of the weighted Davenport constant.

\begin{definition} \cite{ZengYuanDiscrete, WangcommInAlgebra} \ {\sl Let $F$ and $G$ be abelian groups. For any nonempty subset $\Psi$ of ${\rm Hom}(F, G)$, we define the $\Psi-$weighted Davenport constant of $G$, denoted ${\rm D}_{\Psi}(G)$, to be the least positive integer $\ell$ (if it exists, otherwise we let ${\rm D}_{\Psi}(G)=\infty$)
such that every sequence $T$ over $F$ of length
$\ell$ is not $\Psi$-zero-sum free,  i.e.,  there exists $a_1\cdot\ldots\cdot a_k\mid T$ ($k>0$),
such that $\sum\limits_{i=1}^{k}  \psi_i(a_i)=0_G$ for some $\psi_1,\ldots,\psi_{k}\in \Psi$.}
\end{definition}

\section {Groups in which $\mathcal{A}(G)$ is minimal to represent Davenport constant}

\begin{lemma} \cite{EmdeBoas007, Olson1}\label{Lemma Davenport precise value} \ Let $G\cong \mathbb{Z}_{n_1}\oplus \cdots  \oplus \mathbb{Z}_{n_r}$ with $1< n_1 \mid\cdots\mid n_r$. Then
${\rm D}(G)\geq 1 +\sum\limits_{i=1}^r (n_i-1)$. Moreover, equality holds if one of the following conditions holds:
(i) $r\leq 2$; (ii) $G$ is a $p$-group.
\end{lemma}

\begin{lemma} [see \cite{GH}, Proposition 5.1.11] \label{Lemma:DavenportOfQuotient}  Let $G$ be a finite abelian group, and let $H$ be a subgroup of $G$.
Then $D(G)\geq  D(H)+D(G/H)-1$.
\end{lemma}

In what follows of this section, we shall fix one sequence $V$.
Let $H\cong\mathbb{Z}_3^5$ with basis $e_0,e_1,e_2,e_3,e_4$.   Let
\begin{equation}\label{eq:ai}
 a_i=e_i-e_{i-1}-e_{i-2}
\end{equation}
with all indices are read modulo $5$,  i.e., $i\in\mathbb Z/5\mathbb Z$. Define the sequence $V\in \mathcal{F}(H)$ as
\begin{equation}\label{eq:V}
 V=(\prod_{i=0}^{4} a_i)\cdot (\prod_{i=0}^{4}e_i).
\end{equation}
Thus $|V|=10=D(H)-1$.

Then we shall give one lemma to state some properties on the sequence $V$.

\begin{lemma}\label{Lemma:Vatom} Let $H\cong \mathbb{Z}_3^5$, and let $V$ be given as \eqref{eq:ai}. Then the following conclusions hold:

(i) $V\in \mathcal A(H)$;

(ii) $\sum(V)=H$;

(iii)  For every $x\in H$, the sequence $V\cdot x$ contains a minimal zero-sum
subsequence distinct from $V$.
\end{lemma}

\begin{proof}
(i) Summing \eqref{eq:ai} over $i\in\mathbb Z/5\mathbb Z$ gives
$\sum_{i=0}^{4}(a_i+e_i)=0,$ i.e., $V$ is a zero-sum sequence. It remains to show that $V$ is a minimal zero-sum sequence. Take an arbitrary nonempty zero-sum subsequence $U$ of $V$. By \eqref{eq:V}, we may write
$U=\prod_{i=0}^{4}a_i^{\gamma_i} \cdot \prod_{i=0}^{4} e_i^{\beta_i}
 \quad\text{with}\quad
 \gamma_i,\beta_i\in\{0,1\}$.  By \eqref{eq:ai} and by
comparing the coefficient of $e_i$, we derive that
\begin{equation}\label{eq:beta}
 \beta_i\equiv -\gamma_i+\gamma_{i+1}+\gamma_{i+2}
 \pmod 3
\end{equation} for each $i\in\mathbb Z/5\mathbb Z$.
Since $\beta_i\in\{0,1\}$, i.e., the right-hand side in \eqref{eq:beta} cannot be congruent to $2$ modulo $3$, it follows that the $3$-length binary word
$\gamma_i\gamma_{i+1}\gamma_{i+2}$ cannot belong to $\{011, 100\}$, where $i\in \mathbb{Z}\diagup 5\mathbb{Z}$. That is,
the {\bf cyclic} binary word $\gamma_0\gamma_1\cdots\gamma_4$ contains
neither one of $\{011, 100\}$.

Suppose this binary word $\gamma_0\gamma_1\cdots\gamma_4$ is nonconstant, i.e., $\gamma_0\gamma_1\cdots\gamma_4\notin \{00000,11111\}.$ Then every cyclic run of $1$'s would
have length one, since $011$ is forbidden, and every cyclic run of
$0$'s would have length one, since $100$ is forbidden.  Therefore, the {\sl cyclic} word $\gamma_0\gamma_1\cdots\gamma_4$
would be alternate by $0$ and $1$ occurring with the subindex run over $\mathbb{Z}\diagup 5\mathbb{Z}$, which is impossible on an odd cycle of length
five.  Hence, $$\gamma_0\gamma_1\cdots\gamma_4\in \{00000,11111\}.$$

Combined with \eqref{eq:beta}, we conclude that $\beta_0\beta_1\cdots\beta_4=00000$ in case of $\gamma_0\gamma_1\cdots\gamma_4=00000$,  and $\beta_0\beta_1\cdots\beta_4=11111$ in case of $\gamma_0\gamma_1\cdots\gamma_4=11111$. Then either
$U$ is an empty sequence or $U=V$.  Therefore $V$ is a minimal zero-sum sequence, which proves (i).

(ii) The conclusion follows from a  straightforward calculation.

(iii) If $x=0$, the one-term sequence $0$ is the desired minimal zero-sum subsequence distinct from $V$.

Suppose $x\neq 0$ and $x\notin{\rm supp}(V)$.  By
Conclusion (ii), there is a subsequence $U\mid V$ with
$\sigma(U)=-x$.  Since $x\neq 0$, the subsequence $U$ is nonempty and $U\neq V$.
Take a minimal zero-sum subsequence $W$ of $x\cdot U$. Then $W$ is the desired subsequence distinct from $V$, done.

Hence, it remains to consider $x\in{\rm supp}(V)$.
Since the ten terms in \eqref{eq:V}
are all distinct, then $x$ is either equal to $e_i$ or $a_i$ for some $i\in\mathbb Z/5\mathbb Z$.  In case that $x=e_i$ or $x=a_i$, we can check that $e_i^2a_ie_{i-1}e_{i-2}$ or $a_i^2a_{i-1}e_ie_{i-1}e_{i+2}$ is the desired minimal zero-sum subsequence of $V\cdot x$ which is distinct from $V$, respectively. This proves (iii).
\end{proof}

\begin{lemma}\label{Lemma:lifting}
Let $H$ be a subgroup of a finite abelian group $G$, and let
$W\in\mathcal{A}(H)$.  Then we have $d_{\mathcal A(G)\setminus\{W\}}(G)=D(G)$ provided that the following three conditions hold:

(i) $|W|=D(H)-1$;

(ii) $\Sigma(W)=H$;

(iii) For every $x\in H$, the sequence $W\cdot x$ contains a minimal zero-sum subsequence distinct
      from $W$.
\end{lemma}

\begin{proof}
The sequence $W$, viewed as a sequence over $G$, remains a minimal zero-sum sequence.  By Remark \ref{remark}, we have
$d_{\mathcal A(G)\setminus\{W\}}(G)\geq  d_{\mathcal A(G)}(G)=D(G)$, so it remains to prove the opposite inequality.

Let $T$ be a sequence over $G$ of length $D(G)$.  It suffices to show that $T$ contains a minimal zero-sum subsequence distinct from $W$.
If $W\nmid T$, the conclusion follows immediately. Hence, we suppose that $W\mid T$ and we may
write $$T=W\cdot R.$$
Since $|W|=D(H)-1$, it follows from Lemma \ref{Lemma:DavenportOfQuotient} that
$$|R|=|T|-|W|=D(G)-(D(H)-1)\geq D(G/H).$$

If $R$ contains a term $x\in H$, then Condition (iii), applied to the
subsequence $W\cdot x$, gives a minimal zero-sum subsequence distinct from $W$. Hence, we assume that every term of $R$ lies outside $H$.

Let
$\pi:G\to G/H$ be the canonical epimorphism of $G$ onto the quotient group $G/H$.
Since $|\pi(R)|=|R|\geq D(G/H)$, the sequence
$\pi(R)\in \mathcal{F}(G/H)$ has a nonempt zero-sum subsequence. Then we take a {\bf nonempty} subsequence $Y$ of $R$ with $\sigma(Y)\in H$ such that  $|Y|$ is {\bf minimal}.
By Condition (ii), we may take a subsequence $U$ of $W$ with $\sigma(U)=-\sigma(Y)$ such that $|U|$ is {\bf minimal}. Then $U\cdot Y$ is
a nonempty zero-sum subsequence of $T$. Since all terms of $U$ belongs to $H$ and all terms of $Y$ lies outside $H$, it follows from the minimality of both $|U|$ and $|Y|$ that $U\cdot Y$ is a minimal zero-sum subsequence of $T$. Since $Y$ is nonempty, then $U\cdot Y$ is distinct from $W$, completing the proof of the lemma.
\end{proof}

Now we are in a position to prove our main theorem for this section.

\begin{theorem}\label{theorem in group exp(G)geq 6} \ Let $G\cong \mathbb{Z}_3^r$ with $r\geq 1$. Then $\mathcal A (G)$ is a minimal set with respect to the Davenport constant ${\rm D}(G)$ if and only if $r\in [1,4]$.
\end{theorem}

\begin{proof} The case of $r\in \{1,2\}$ follows from  Theorem A. If $r\in \{3,4\}$, by a straightforward calculation, we can show that $\mathcal A (G)$ is a minimal set with respect to ${\rm D}(G)$.
Now suppose that $r\geq 5$. Write
$G=\langle e_1,\ldots,e_r\rangle\cong\mathbb Z_3^r$
and put
$H=\langle e_1,\ldots,e_5\rangle.$
Then
$H\cong\mathbb Z_3^5$.
Let $V\in\mathcal A(H)$ be the sequence constructed as in  \eqref{eq:V}.
 By Lemma \ref{Lemma:Vatom}, we check that the sequence $V$ satisfies Conditions (i)--(iii) in Lemma \ref{Lemma:lifting} with $W=V$.  Then applying Lemma \ref{Lemma:lifting}, we have that $d_{\mathcal{A}(G)\setminus \{V\}}(G)={\rm D}(G)$. \end{proof}

 Combined with Theorem A, we have the following Theorem immediately which classified all finite abelian groups $G$ such that $\mathcal A (G)$ is minimal.

\begin{theorem} \label{theorem in group exp(G)geq 6} \
Let $G$ be a finite nonzero abelian group. Then $\mathcal A(G)$ is
a minimal set with respect to the Davenport constant $D(G)$ if and
only if one of the following conditions holds:
(i) $G\cong \mathbb Z_2^r$ with $r\geq 1$;
(ii) $G\cong \mathbb Z_3^r$ with $1\le r\leq 4$;
(iii) $G\cong \mathbb Z_4$ or $G\cong \mathbb Z_5$.
\end{theorem}

\section{Finiteness of Weighted Davenport constant}

We begin this section with some terminologies. Let $\{F_i:i\in I\}$ be a family of abelian groups, and let $G$ be an
abelian group. Let
$\psi_i:F_i\longrightarrow G \qquad (i\in I)$
be homomorphisms. By the universal property of the direct sum, there
exists a unique homomorphism
$
  \psi:\bigoplus\limits_{i\in I}F_i\longrightarrow G
$
such that
\begin{equation}\label{equation universal map}
  \psi\circ\iota_i=\psi_i
  \qquad\text{for every }i\in I,
\end{equation}
where
$
  \iota_i:F_i\longrightarrow\bigoplus\limits_{t\in I}F_t
$
is the canonical injection. We call $\psi$ the \emph{universal
homomorphism induced by the family $\{\psi_i:i\in I\}$}, denoted
$$\psi=\operatorname{ind}[\{\psi_i:i\in I\}].$$
Now let $n\geq 1$ and let $F_1=\cdots=F_n=F$. We identify
$\bigoplus\limits_{i=1}^n F_i$ with $F^n$. For any family of homomorphism
$\{\psi_i\}_{i\in [1,n]}$, also denoted as $(\psi_1,\ldots,\psi_n)$, we can write
$\operatorname{ind}[\{\psi_i:i\in [1,n]\}]$ as $\operatorname{ind}[(\psi_1,\ldots,\psi_n)].$
Therefore, for any $\boldsymbol{\eta}=(\psi_1,\ldots,\psi_n)\in \operatorname{Hom}(F,G)^n$ and any $(x_1,\ldots,x_n)\in F^n$,
we have
\begin{equation}\label{eq:induced-explicit}
\operatorname{ind}[\boldsymbol{\eta}]\left((x_1,\ldots,x_n)\right)=\operatorname{ind}[\boldsymbol{\eta}]\left(\sum\limits_{i=1}^n \iota_i(x_i)\right)=\sum\limits_{i=1}^n  \operatorname{ind}[\boldsymbol{\eta}]\left(\iota_i(x_i)\right)=\sum\limits_{i=1}^n \psi_i(x_i).
\end{equation}
Note that if $(\psi_1,\ldots,\psi_n)\neq (\mathbf 0,\ldots,\mathbf 0)$, then the induced universal
homomorphism $\operatorname{ind}[(\psi_1,\ldots,\psi_n)]$ is a nonzero homomorphism of $F^n$ to $G$.

Then we shall prove our main theorem for this section.

\begin{theorem}\label{thm:kernel-cover}
Let $F$ and $G$ be abelian groups, let
$\varnothing\neq\Psi\subseteq\operatorname{Hom}(F,G)\setminus\{\mathbf 0\},$
and let $n\ge 1$. Then the following two statements are equivalent.
\begin{enumerate}
  \item[\rm (i)] ${\rm D}_{\Psi}(G)\le n$.

  \item[\rm (ii)] The group $F^n$ is covered by the kernels of the
  nonzero universal homomorphisms induced by the tuples in
  $(\Psi\cup\{\mathbf 0\})^n$; more precisely,
  $$F^n=\bigcup_{\boldsymbol{\eta}\in
      (\Psi\cup\{\mathbf 0\})^n\setminus\{(\mathbf 0,\ldots,\mathbf 0)\}}
    \ker\bigl(\operatorname{ind}[\boldsymbol{\eta}]\bigr).$$
\end{enumerate}

Consequently,
\begin{equation}\label{eq:D-minimum}
 {\rm D}_{\Psi}(G)
 =
 \min\left\{
 n\in\mathbb N:
 F^n
 =
 \bigcup_{\boldsymbol{\eta}\in
   (\Psi\cup\{\mathbf 0\})^n\setminus
   \{(\mathbf 0,\ldots,\mathbf 0)\}}
 \ker\bigl(\operatorname{ind}[\boldsymbol{\eta}]\bigr)
 \right\},
\end{equation}
where the minimum of the empty set is understood to be $\infty$.
\end{theorem}

\begin{proof}
Assume first that ${\rm D}_{\Psi}(G)\le n$. We take an arbitrary element
$(x_1,\ldots,x_n)\in F^n.$
Consider the sequence
$T=x_1\cdot\ldots\cdot x_n$
over $F$. By the definition of ${\rm D}_{\Psi}(G)$, the sequence $T$
has a nonempty $\Psi$-weighted zero-sum subsequence. Hence there exist
a {\bf nonempty} set $I\subseteq[1,n]$ and homomorphisms
$\psi_i\in\Psi \ (i\in I)$
such that
$\sum_{i\in I}\psi_i(x_i)=0_G.$ For each $j\in[1,n]\setminus I$, we set $\psi_j=\mathbf 0$. Denote
$\boldsymbol{\eta}=(\psi_1,\ldots,\psi_n).$
Since $I\neq\varnothing$, the tuple $\boldsymbol{\eta}$ is not the
zero tuple, i.e., $\boldsymbol{\eta}\in
   (\Psi\cup\{\mathbf 0\})^n\setminus
   \{(\mathbf 0,\ldots,\mathbf 0)\}$. By \eqref{eq:induced-explicit}, we derive that
$\operatorname{ind}[\boldsymbol{\eta}]\left((x_1,\ldots,x_n)\right)=\sum_{i=1}^n\psi_i(x_i)=\sum_{i\in I}\psi_i(x_i)=0_G.$ That is,
$(x_1,\ldots,x_n)\in
  \ker\bigl(\operatorname{ind}[\boldsymbol{\eta}]\bigr).$
By the arbitrariness of choosing $(x_1,\ldots,x_n)$ from $F^n$, Statement \emph{(ii)} holds. This proves that Statement (i) implies Statement (ii).

Now we show that Statement (ii) implies Statement (i). Suppose that (ii) holds.  Let
$T=x_1\cdot\ldots\cdot x_n$
be an arbitrary sequence over $F$ of length $n$.  Since $(x_1,\ldots,x_n)\in F^n$, it follows that there
exists
$\boldsymbol{\eta}=(\psi_1,\ldots,\psi_n)
  \in
  (\Psi\cup\{\mathbf 0\})^n
  \setminus\{(\mathbf 0,\ldots,\mathbf 0)\}$
such that
$(x_1,\ldots,x_n)
  \in
  \ker\bigl(\operatorname{ind}[\boldsymbol{\eta}]\bigr).$
Set
$J=\{j\in[1,n]:\psi_j\neq\mathbf 0\}.$
Since $\boldsymbol{\eta}$ is not the zero tuple, we have $J\neq
\varnothing$. Note that $\psi_j\in\Psi$ for each $j\in J$, and that
$0_G=\operatorname{ind}[\boldsymbol{\eta}]
      \left((x_1,\ldots,x_n)\right)=
      \sum_{i=1}^n\psi_i(x_i) =
      \sum_{j\in J}\psi_j(x_j).$
Thus
$\prod\limits_{j\in J}x_j$
is a nonempty $\Psi$-weighted zero-sum subsequence of $T$. By the arbitrariness of choosing $T\in \mathcal{F}(F)$ with $|T|=n$, it follows that
${\rm D}_{\Psi}(G)\le n.$

The formula \eqref{eq:D-minimum} follows immediately from the
equivalence of (i) and (ii).
\end{proof}

Now we show that Theorem \ref{thm:kernel-cover} implies that Theorem B.
In fact, we just need to show  that if there exists some $n>0$ such that $F^n$ is covered by the kernels of all nonzero induced universal homomorphisms, then we have a finite cover of $F$ given as Theorem B (ii). Precisely, we show the following.

\begin{prop}\label{cor:finite-coset-cover}
Let $F$ and $G$ be abelian groups, and let $\Psi$ be a nonempty
{\bf finite} subset of $\operatorname{Hom}(F,G)$. Suppose that there exists some $n>0$ such that $F^n=\bigcup_{\boldsymbol{\eta}\in
      (\Psi\cup\{\mathbf 0\})^n\setminus\{(\mathbf 0,\ldots,\mathbf 0)\}}
    \ker\bigl(\operatorname{ind}[\boldsymbol{\eta}]\bigr).$  then $F$ has a finite cover by cosets of $\ker(\tau)$,
where $\tau\in\Psi$.
\end{prop}

\begin{proof}
If $\mathbf 0\in\Psi$, then $F=\ker(\mathbf 0)$, and there is nothing
to prove. Thus assume that $\mathbf 0\notin\Psi$. Put
$d={\rm D}_{\Psi}(G)$.
Choose a $\Psi$-zero-sum free sequence
$T=b_1\cdots b_{d-1}\in \mathcal F(F)$ of length $d-1$.  Define the finite set
$$S=
\left\{
\sum_{i=1}^{d-1}\psi_i(b_i):
\psi_i\in\Psi\cup\{\mathbf 0\}
\right\}.$$
Take an arbitrary element $x$ of $F$. Since $(b_1,\ldots,b_{d-1},x)\in F^d$, it follows that
$(b_1,\ldots,b_{d-1},x)\in \ker(\operatorname{ind}[(\tau_1,\ldots,\tau_d)])$
for some
$(\tau_1,\ldots,\tau_d)
\in(\Psi\cup\{\mathbf 0\})^d
\setminus\{(\mathbf 0,\ldots,\mathbf 0)\}.$ By \eqref{eq:induced-explicit}, we have
$\sum_{i=1}^{d-1}\tau_i(b_i)+\tau_d(x)=0.$
Note that $\tau_d\neq\mathbf 0$. Otherwise, the nonzero maps among
$\tau_1,\ldots,\tau_{d-1}$ would give a nonempty
$\Psi$-weighted zero-sum subsequence of $T$, contradicting the choice
of $T$. Thus $\tau_d\in\Psi$, and $\tau_d(x)\in -S.$ By the arbitrariness of $x\in F$, we have $F=\bigcup\limits_{\tau\in\Psi}\tau^{-1}(-S)=\bigcup\limits_{\tau\in\Psi}\bigcup\limits_{s\in S}\tau^{-1}(-s)$.
Since both $\Psi$ and $S$ are finite, this is a finite union.
Moreover, every set $\tau^{-1}(-s)$ is a coset of
$\ker\tau$.  Then $F$ has a finite cover by cosets of kernels $\ker(\tau)$ with
$\tau\in\Psi$.
\end{proof}

The following example shows that, for infinite $\Psi$, the kernel
cover given by Theorem \ref{thm:kernel-cover} need not admit a finite subcover, even when
${\rm D}_{\Psi}(G)$ is finite.

\begin{exam}\label{ex:no-finite-kernel-subcover}
Let $F=G=\mathbb Z$
and let $\Psi=\operatorname{End}(\mathbb Z)\setminus\{0\}$. Then $\Psi=\{\psi_k:k\in\mathbb Z\setminus\{0\}\}$
where $\psi_k:z\mapsto kz$ for any $z\in \mathbb{Z}$. Then
${\rm D}_{\Psi}(\mathbb Z)=2.$
However, for every $n\ge 1$, the group $\mathbb Z^n$ cannot be
covered by finitely many kernels of {\bf nonzero} universal homomorphisms
induced by tuples in
$(\Psi\cup\{\mathbf 0\})^n=(\operatorname{End}(\mathbb Z))^n.$
\end{exam}

\begin{proof}
It is easy to show that ${\rm D}_{\Psi}(\mathbb Z)=2,$  for which one can also refer to the argument of Example 5.8 in \cite{WangcommInAlgebra}.

It remains to prove the assertion about finite kernel covers. Suppose to the contrary that there exists some integer $n>0$, and there exists $\boldsymbol{\eta}_1,\ldots,\boldsymbol{\eta}_s\in (\operatorname{End}(\mathbb Z))^n\setminus \{\boldsymbol{0}\}$ with $s>0$ such that $\mathbb{Z}^n=\bigcup\limits_{i=1}^s \ker(\operatorname{ind}[\boldsymbol{\eta}_i])$.  Note that $\boldsymbol{\eta}_i=(\psi_{k^{(i)}_1}, \ldots,\psi_{k^{(i)}_n})$ where   $(k^{(i)}_1,\ldots,k^{(i)}_n)$ is a nonzero vector of $\mathbb{Z}^n$ for each $i\in [1,s]$. By \eqref{eq:induced-explicit}, we see that the induced universal homomorphism $\operatorname{ind}[(\psi_{k_1},\ldots,\psi_{k_n})]:\mathbb Z^n\longrightarrow \mathbb Z$
is given by
$$\operatorname{ind}[(\psi_{k_1^{(i)}},\ldots,\psi_{k_n^{(i)}})]:  \ \ \ (z_1,\ldots,z_n)\mapsto k_1^{(i)}z_1+\cdots+k_n^{(i)}z_n.$$
Define the polynomial
$$P^{(i)}(x)
=
k^{(i)}_1
+
k^{(i)}_2x
+
\cdots
+
k^{(i)}_n x^{n-1}
\in\mathbb Z[x],$$ where $i\in [1,s]$.
 For each $i\in [1,s]$, we see that $P^{(i)}(x)$ is a nonzero polynomial, and thus,
it has only finitely many integer roots. Then we can find some $u\in \mathbb{Z}$ such that $u$ is not a root of $P^{(i)}(t)$ for all $i\in [1,s]$, i.e., $k^{(i)}_1
+
k^{(i)}_2u
+
\cdots
+
k^{(i)}_n u^{n-1}\neq 0$. Thus, $(1,u,\ldots,u^{n-1})\notin \bigcup\limits_{i=1}^s \ker(\operatorname{ind}[\boldsymbol{\eta}_i])$, a contradiction.
\end{proof}

Compared with the finite weights set, one question is basic: Suppose $D_{\Psi}(G)$ is finite. In Theorem \ref{thm:kernel-cover} (ii), when do we have a cover of $F^n$ by finitely many kernels?  In the following theorem, we shall show that the existence of finite cover by kernels of induced universal homomorphism is equivalent to the finite reduction property of the weights set.

\begin{prop}
Let $F$ and $G$ be abelian groups, let
$\varnothing\neq\Psi\subseteq \operatorname{Hom}(F,G)$, and let $n\ge 1$.
Then the following two conclusions are equivalent:

(i) The canonical kernel cover
$$F^n=
\bigcup_{\boldsymbol{\eta}\in
(\Psi\cup\{\mathbf 0\})^n
\setminus\{(\mathbf 0,\ldots,\mathbf 0)\}}
\ker(\operatorname{ind}[\boldsymbol{\eta}])$$
has a finite subcover;

(ii) There exists a finite subset $\Psi_0\subseteq \Psi$ such that
$$D_{\Psi_0}(G)\le n.$$

In particular, if $n={\rm D}_{\Psi}(G)<\infty$, then the above
conditions are equivalent to the existence of a finite subset
$\Psi_0\subseteq\Psi$ such that ${\rm D}_{\Psi_0}(G)={\rm D}_{\Psi}(G).$
\end{prop}

\begin{proof}
Suppose first that the canonical kernel cover has a finite subcover.
Then there exist finitely many tuples
$\boldsymbol{\eta}^{(1)},\ldots,\boldsymbol{\eta}^{(m)}
\in
(\Psi\cup\{\mathbf 0\})^n
\setminus\{(\mathbf 0,\ldots,\mathbf 0)\}$
such that
$F^n=
\bigcup_{j=1}^{m}
\ker(\operatorname{ind}[\boldsymbol{\eta}^{(j)}]).$
Let $\Psi_0$ be the finite set of all nonzero homomorphisms appearing
as coordinates of the tuples
$\boldsymbol{\eta}^{(1)},\ldots,\boldsymbol{\eta}^{(m)}$.
Then $F^n=
\bigcup_{\boldsymbol{\eta}\in
(\Psi_0\cup\{\mathbf 0\})^n
\setminus\{(\mathbf 0,\ldots,\mathbf 0)\}}
\ker(\operatorname{ind}[\boldsymbol{\eta}]).$
By Theorem \ref{thm:kernel-cover}, we derive that
$D_{\Psi_0}(G)\leq n.$

Conversely, suppose that $\Psi_0\subseteq\Psi$ is finite and
$D_{\Psi_0}(G)\le n$. Again by  Theorem \ref{thm:kernel-cover},
$F^n=
\bigcup_{\boldsymbol{\eta}\in
(\Psi_0\cup\{\mathbf 0\})^n
\setminus\{(\mathbf 0,\ldots,\mathbf 0)\}}
\ker(\operatorname{ind}[\boldsymbol{\eta}]).$
Since $\Psi_0$ is finite, this is a finite subcover of the canonical
$\Psi$-kernel cover.
\end{proof}

From Theorem~\ref{thm:kernel-cover}, if $D_{\Psi}(G)\le n$, then the
canonical family $\mathcal K_n(\Psi)$ covers $F^n$.  If this cover
has a finite subcover, then only finitely many homomorphisms from
$\Psi$ are involved, and consequently $D_{\Psi}(G)$ is already
attained by a finite subset $\Psi_0$ of $\Psi$.
This motivates the following
kernel-cover compactness property.

\begin{definition}
Let $F$ and $G$ be abelian groups, let
$\varnothing\neq\Psi\subseteq\operatorname{Hom}(F,G)$, and let $n\ge1$. Put
$$
\mathcal K_n(\Psi)
=
\left\{
\ker(\operatorname{ind}[\boldsymbol{\eta}]):
\boldsymbol{\eta}\in
(\Psi\cup\{\mathbf 0\})^n
\setminus\{(\mathbf 0,\ldots,\mathbf 0)\}
\right\}.
$$
We say that $F^n$ is \emph{$\Psi$-kernel-cover compact} (at level $n$)
if every cover of $F^n$ by members of $\mathcal K_n(\Psi)$ has a
finite subcover.

Furthermore, we say that $F^n$ is {\bf uniformly} \emph{$\Psi$-kernel-cover compact} provided that for any $X\subseteq F^n$, every cover of $X$ by members of $\mathcal K_n(\Psi)$ has a
finite subcover.
\end{definition}

\begin{remark} By definition, uniform $\Psi$-kernel-cover compactness implies
$\Psi$-kernel-cover compactness.
\end{remark}

\begin{lemma}\label{lemma:sometau} Let $F$ and $G$ be abelian groups, let
$\Psi\subseteq\operatorname{Hom}(F,G)$. If there exists $\tau\in \Psi$ such that $F/\ker(\tau)$ is finite, then $D_{\Psi}(G)<\infty$.
\end{lemma}

\begin{proof} By definition, it is not hard to show that $D_{\Psi}(G)\leq D(F/\ker(\tau))\leq |F/\ker(\tau)|$. One can also refer Lemma 5.5 in \cite{WangcommInAlgebra}.
\end{proof}

Next we shall give a sufficient condition such that $F^n$ is uniformly $\Psi$-kernel-cover compact which is stated as the following proposition.

\begin{prop}
\label{prop:finite-effective-quotient}
Let $F$ and $G$ be abelian groups, let
$\varnothing\neq\Psi\subseteq\operatorname{Hom}(F,G)$, and put $N_{\Psi}
=
\bigcap\limits_{\psi\in\Psi}\ker(\psi).$
If $F/N_{\Psi}$ is finite, then $D_{\Psi}(G)$ is finite and $F^n$ is uniformly $\Psi$-kernel-cover compact for every $n\geq D_{\Psi}(G)$.
\end{prop}

\begin{proof} Since $[F:N_{\Psi}]<\infty$, it follows that $[F:\ker(\tau)]<\infty$ for any $\tau\in \Psi$. By Lemma \ref{lemma:sometau}, we have that $D_{\Psi}(G)<\infty$. It suffices to show the uniform $\Psi$-kernel-cover compactness of $F^n$ for any integer $n\geq D_{\Psi}(G)$.

Let $X$ be an arbitrary subset of $F^n$. Suppose that
\begin{equation}\label{equation F^n=union}
X\subseteq \bigcup_{\boldsymbol{\gamma}\in\Gamma}\ker(\boldsymbol{\gamma})
\end{equation}
where
$\Gamma\subseteq
\left\{
\operatorname{ind}[\boldsymbol{\eta}]:
\boldsymbol{\eta}\in
(\Psi\cup\{\mathbf 0\})^n
\setminus\{(\mathbf 0,\ldots,\mathbf 0)\}
\right\}.$

We first show that
\begin{equation}\label{eq:Npsisubset}
N_{\Psi}^n\subseteq \ker(\boldsymbol{\gamma})
\mbox{ for every } \boldsymbol{\gamma}\in\Gamma.
\end{equation}
Fix $\boldsymbol{\gamma}\in\Gamma$. Then $\boldsymbol{\gamma}=\operatorname{ind}[(\psi_1,\ldots,\psi_n)]$ with $(\psi_1,\ldots,\psi_n)\in
(\Psi\cup\{\mathbf 0\})^n
\setminus\{(\mathbf 0,\ldots,\mathbf 0)\}.$
Take arbitrary
$\boldsymbol y=(y_1,\ldots,y_n)\in N_{\Psi}^n.$
That is, for each $i\in [1,n]$ we have $y_i\in N_{\Psi}$. Since
$N_{\Psi}=\bigcap_{\psi\in\Psi}\ker\psi,$
we have $y_i\in\ker\psi_i$ whenever $\psi_i\in\Psi$. If
$\psi_i=\mathbf 0,$ then also $\psi_i(y_i)=0$. Hence,
$\boldsymbol{\gamma}(\boldsymbol y)=\sum_{i=1}^n\psi_i(y_i)=0.$
This proves \eqref{eq:Npsisubset}.

Consequently, $\ker(\boldsymbol{\gamma})$ is a union of cosets of $N_{\Psi}^n$ for every $\gamma\in\Gamma$.
Since the quotient group $F/N_{\Psi}$ is finite, it follows that $F^n/N_{\Psi}^n\cong (F/N_{\Psi})^n$ is finite, and so the image of $X$ in $F^n/N_\Psi^n$ is finite.   Then we can choose finitely many elements
$\boldsymbol{b}_1,\ldots,\boldsymbol{b}_s\in X$ such that
\begin{equation}\label{equation F^ncoveri=1tos}
X\subseteq \bigcup\limits_{j=1}^s (\boldsymbol{b}_j+N_{\Psi}^n).
\end{equation}
 By \eqref{equation F^n=union}, for each $j\in [1,s]$ we can take an induced universal homomorphism $$\boldsymbol{\gamma}_j\in\Gamma$$ such that
 \begin{equation}\label{eq:bjinetaj}
 \boldsymbol{b}_j\in\ker(\boldsymbol{\gamma}_j).
 \end{equation}

Let $\boldsymbol x$ be an arbitrary element of $X$. By \eqref{equation F^ncoveri=1tos}, we see that $\boldsymbol x-\boldsymbol{b}_k\in N_{\Psi}^n$ for some $k\in [1,s]$. Then it follows from \eqref{eq:Npsisubset} that $\boldsymbol x-\boldsymbol{b}_k\in \ker(\boldsymbol{\gamma}_k)$. Combined with \eqref{eq:bjinetaj}, we have $\boldsymbol x\in \ker(\boldsymbol{\gamma}_k)$. By the arbitrariness of choosing $\boldsymbol x$, we derive that $X\subseteq \bigcup_{j=1}^s\ker(\boldsymbol{\gamma}_j)$, and thus prove the conclusion.
\end{proof}

\bigskip

\noindent {\bf Acknowledgements}

\noindent This work is supported by NSFC (grant no. 12371335, 12271520).

\end{document}